\documentclass[english,a4paper,12pt,final,reqno,doublespace]{amsart}
\usepackage{standardpaketet}
\usepackage{float}
\usepackage{graphicx}
\usepackage{subfig}

   \newcommand{\al}{\alpha}

    \newcommand {\bC} {\mathbb C}
    
    \newcommand {\bZ} {\mathbb Z}
    \newcommand{\bN}{\mathbb N}

\newtheorem*{theo+}           {{\bf Theorem}}
\newtheorem*{theorem-non}    {{\bf Theorem}}
\newtheorem{prop+}           {Proposition}
\newtheorem{coro+}           {Corollary}
\newtheorem{lemm+}           {Lemma}

\newtheorem{defi+}           {Definition}

\newtheorem{not+}            {Notation}

\newtheorem{rema+}           {Remark}

\begin{document}
\title[Zeros  for Gauss hypergeometric polynomials]{Zeros of a certain class of Gauss hypergeometric polynomials}

\author{Addisalem Abathun }
\address{Department of Mathematics, Addis Ababa University and Stockholm University}
\email{addisaa@math.su.se}
\author{Rikard B\o gvad}
\address{Department of Mathematics, Stockholm University}
\email{rikard@math.su.se}
\date \today
\begin{abstract}In this paper, we give results that partially prove a conjecture which was discussed in our previous work \cite{AARB}.  More precisely, we prove that as   $n\to \infty,$ the zeros of the polynomial$${}_{2}\text{F}_{1}\left[ \begin{array}{c} -n, \al n+1\\ \al n+2 \end{array} ; \begin{array}{cc} z \end{array}\right]$$
cluster on a certain curve defined as a part of a level curve of an explicit harmonic function. This generalizes work by Boggs, Driver,  Duren et. al \cite{KBD0,KDPD1,LDJG}, to a complex parameter $\alpha$.

\end{abstract}

\maketitle

\section{Introduction}
The generalized hypergeometric function ${}_{A}\text{F}_{B}$  with A numerator and B denominator parameters is defined by
\begin{equation}
{}_{A}\text{F}_{B}\left[ \begin{array}{c} (a)_A \\ (b)_B \end{array} ; \begin{array}{cc} z \end{array}\right]=
{}_{A}\text{F}_{B}\left[ \begin{array}{c} a_1,a_2,\ldots,a_A \\ b_1,b_2,\ldots,b_B \end{array} ; \begin{array}{cc} z \end{array}\right]
=\sum_{k=0}^{\infty}\frac{\prod_{j=1}^{j=A}(a_{j})_{k}}{\prod_{j=1}^{j=B}(b_{j})_{k}}\frac{z^k}{k!}\label{eq1} \end{equation}
 where $a_i\in\bC$, $b_j\in\bC\setminus \bZ_0^-$, $1\leq i\leq A$, $1\leq j\leq B$ and

 $(\alpha)_{k}=\alpha(\alpha+1)\ldots(\alpha+k-1)=\frac{\Gamma(\alpha+k)}{\Gamma(\alpha)}$ is the Pochhammer symbol.
    \\If any of the numerator parameters is a negative integer, say $a_1=-n,n\in\mathbb{N}$, the series terminates and reduces to a polynomial of degree $n$, called a {\it generalized hypergeometric polynomial}.
We are, in this note,  interested in  the asymptotic zero distribution of the polynomial coming from an ordinary hypergeometric polynomial
    \begin{equation}
   \label{hyp}
    p_n(z)={}_{2}\text{F}_{1}(-n,a_2(n);b_1(n);z)
   \end{equation}
   with complex parameters, dependent linearly on $n$, as $n\to \infty$.

 Based on experimental evidence  and results by previous authors, we made in \cite{AARB} a conjectures on the asymptotic behaviour of zeros of a certain class of such hypergeometric polynomials.   The  purpose of this note is to partially prove this conjecture.
   \\Our result extends results by earlier authors. We will describe them.

  In 1999, K.Driver and P.Duren studied the zeros of the hypergeometric polynomial
  $F(-n, kn + 1; kn + 2; z)$  for integers $k,n>0.$ They used the  Euler integral representation together with a general theorem of Borwein and Chen and showed the following result.
  \begin{theorem}\label{Th1} Let $k ~\text{and}~  n \in \bN.$  Then the zeros of the hypergeometric polynomial
  $$F(-n, kn + 1; kn + 2; z)$$
cluster on the loop of the lemniscate
 $$\{z:|z^k(z-1)|=\frac{k^k}{(k+1)^{k+1}};~Re(z)>\frac{k}{k+1}\}$$
 as $n \to \infty$.
 \begin{proof}
 See, \cite[Theorem 1]{KDPD1}

 \end{proof}
 \end{theorem}
 In 2001, K.Boggs and P.Duren gave the following extension of Theorem \ref{Th1}
\begin{theorem}
For arbitrary $k > 0$ and $\emph{l}>0$  the zeros of the hypergeometric polynomial $$F(-n, kn + l+1; kn +l+ 2; z)$$ cluster on the loop of lemniscate $$|z^k(z-1)|=\frac{k^k}{(k+1)^{k+1}}~ \text{with}~ Re(z)>\frac{k}{k+1}$$.
\end{theorem}
\begin{proof}
 See, \cite[Theorem 2]{KDPD1}

 \end{proof}
In this paper, we continue to study the asymptotic properties of zeros of hypergeometric function $F(-n, kn + 1; kn + 2; z)$ to allow $k$ to be any complex number with non negative real part. We have the following theorem.

\begin{theorem-non} {\it
 Let $\al = \eta + i \zeta $  where  $\eta > 0$ and $\zeta \neq 0$. The zeros of the hypergeometric polynomials
 $$p_n(z)={}_{2}\text{F}_{1}\left[ \begin{array}{c} -n, \al n + 1 \\ \al n+2 \end{array} ;
 \begin{array}{cc} z \end{array}\right],$$ asymptotically cluster on the level curve
  \begin{equation*}\label{lami}
{ |z^\al(1-z)|}=\vert(\frac{\al}{(\al+1)})^{\al}(\al+1)^{-1}\vert.
 \end{equation*}
}
 \end{theorem-non}

 In  \cite{AARB} it was proved, using indirect methods, that a convergent sequence of zeroes(described as a limit of corresponding measures has to cluster along a level curve of the function
 ${ |z^\al(1-z)|}$. This result identifies which level curve. As we will see, it is the unique level curve through the saddle point of the function. As seen from the results quoted, the zeroes will not cluster on the whole level curve. We prove in Lemma \ref{lemma:principalbranch} and Proposition \ref{lemma:zero free} that there is a certain region, containing the left half-plane that is zero-free. In  \cite{A}  we also obtain the density of the zeros. Some pictures of the accumulation of the zeroes may also be found there.

  The proof  uses, following the references above,   Euler integral representation of these hypergeometric polynomials and the saddle point method.
   For convenience we will give a short description below of the saddle point method and the Euler integral representation of Gauss hypergeometric functions.

   \section{ Integral representation of Hypergeometric function and the Saddle point method}
   \subsection{Integral Representation of Hypergeometric function}
The Gauss hypergeometric series ${}_{2}\text{F}_{1}(a,b;c;z)$ has the following integral representation due to Euler See \cite[Theorem 2.2.1]{A}. If $Re(c)>Re(b)>0,$ then
\begin{equation}
{}_{2}\text{F}_{1}\left[ \begin{array}{c} a,b \\ c \end{array} ; \begin{array}{cc} z \end{array}\right]
=\frac{\Gamma(c)}{\Gamma(b)\Gamma(c-b)}\int_0^1t^{b-1}(1-t)^{c-b-1}(1-tz)^{-a} dt,
 \label{eq2}
 \end{equation}
 in the $z$ plane cut along the real axis from 1 to $\infty.$  Here it is understood that $\arg t= \arg (1-t)=0$ and $(1-zt)^{-a}$ has its principal value.

 Choosing $a=-n$ and $c=b+1$ and $b=\al n+1$ for $n \in \bN$ for an arbitrary complex number $\al,  Re(\al)>0$, we
get
$${}_{2}\text{F}_{1}\left[ \begin{array}{c} -n, \al n+1\\ \al n+2 \end{array} ; \begin{array}{cc} z \end{array}\right]
=(\al n+1)\int_0^1t^{\al n}(1-tz)^{n} dt.$$

 \subsection{The Saddle Point Method }
  We  use the saddle point method \cite{CBSO,NG} to get an asymptotic expansion for the Euler integral of the hypergeometric polynomial $F(-n, \al n+1; \al n+2; z)$, which will help us to understand the asymptotic distribution of the zeros.

  Let $\gamma$ be a contour in the complex plane and the function $\phi$ is
holomorphic in a neighborhood of this contour. We want to understand
the asymptotics, as $n \to \infty$, of  an integral:
$$I_n(z)=\int_{\gamma}e^{n\phi(z)}dz.$$
The idea of the saddle point method
is to deform the contour in such a way that the main contribution to the
integral comes from a neighborhood of a single point.
Consider the derivative $\phi'(z)$ of an analytic function $\phi(z)=u(x,y)+i v(x,y)$. It is an immediate consequence of the Cauchy-Riemann equations that the gradient of $u$ is given(as a complex vector) by $\overline \phi'(z)$. It is orthogonal to the gradient of $v$, and hence if $\phi'(z_0)\neq 0$, there is a unique curve through $z_0$ characterized by both the property that $Im \phi(z)=Im \phi(z_0)$ and that along this curve $Re \phi(z)$ grows fastest. This curve is called a curve of {\it steepest ascent}(or {\it descent,} if we reverse the direction).

A saddle point is a point $t_0$ where $\phi'(t_0)=0$, and if furthermore $\phi''(t_0)\neq 0,$ then it is called a simple saddle point.
Then in a small neighborhood of a simple saddle point $t_0$ the level curve $Im \phi(z)=Im \phi(t_0)$ consists
of two analytic curves that intersect orthogonally at the point $t_0$ and separate
the neighborhood of $t_0$ in four sectors. Seen from $t_0$ there are four curve segments starting at $t_0$, and along two of them $Re \phi(z)$ decreases from $Re \phi(z_0).$ Call these curves $\gamma_1,~\gamma_2$. We will use a simple form of the saddle point method, which can be found in \cite{CBSO, IA}

\begin{proposition}
\label{prop:saddlepoint}
Suppose there is only one simple saddle point $t_0$ and $\gamma_1$ contains the point $a$ and $\gamma_2$ contains the point $b$. Furthermore, we assume that the curves from $a$ to $t_0$ along $\gamma_1$ and from  $t_0$ to $b$ along $\gamma_2$ are finite. Then
$$I_n(z)=\sqrt{\frac{2\pi}{-\phi''(t_0)}}n^{-\frac{1}{2}}e^{n\phi(t_0)}[1+O(n^{-1}))]$$
\end{proposition}

\section{Proof  the Theorem}
We are (following \cite{LDJG}) going to prove the theorem by applying the saddle point method to evaluate the asymptotic expansion of the Euler integral of the polynomial

$${}_{2}\text{F}_{1}\left[ \begin{array}{c} -n, \al n+1\\ \al n+2 \end{array} ; \begin{array}{cc} z \end{array}\right]
=(\al n+1)\int_0^1t^{\al n}(1-zt)^{n} dt,$$  where $\al=\eta + i\zeta,$ $\eta>0$ and $\zeta\neq 0.$ We let
$$p_n=\int_0^1 [g(t)]^ndt,$$

where $g(t)=t^{\al }(1-tz)$ is a function of the complex variable $t$, and we have chosen the ordinary branch of the complex logarithm(which specializes to the real logarithm on $[0,1]$). The multivalued function $g(t)$ vanishes at $0$ and $1/z$, and has $t=0$ as the only branch point.

\subsection{The saddlepoint}
To find the saddle points consider $\phi(t)=\al\log t + \log (1-zt)$, for the principal branch determination of the logarithm. For each fixed $z\neq 0$, the function $\phi$ has two branch points, at $t=0$ and $t=\frac{1}{z}$. Since
\begin{equation}
\label{eq:dlog}
\phi'(t)=\frac{{\al-zt((\al+1))}}{t(1-zt)}
\end{equation}
we see that $\phi'(t)=0$ if and only if  $t=t_0:=\frac{\al}{(\al+1)z}$, this is the only saddle point of $\phi$.
Since $\phi''(t)=-(1+\al )^3z^2/\al $, it is a simple saddle point if $\al \neq -1$ and $z\neq 0$.(The first condition is true by assumption, and it is easy to see from the description of the hypergeometric polynomials that $z=0$ is not a zero of any $p_n$.) The fact that there is only one saddle point, makes it easy to exploit the local description for some global information. Note that the fact that the derivative $\phi'$ is defined in the whole $\mathbb C$ by (\ref{eq:dlog}) shows that paths of steepest ascent/descent are unambiguously defined, irrespective of choices of branches of logarithms. This follows since the gradient(considered as a complex number) at $t$ of $Re \phi$ by the Cauchy-Riemann equations is just $\overline{\phi'(t)}$.
 \begin{lemma}
 \label{lemma:steepest descent}  Let $t_0=\frac{\al}{(\al+1)z}$ as above be the saddle point of $g(t)$, where $\al \neq -1$ and $z\neq 0$.
 \begin{enumerate}
   \item[i)]  There are exactly two paths, of steepest descent,  starting at the saddle point $t_0$ along which $Re \phi(t)$ decreases,  one,  say $\gamma_1$ goes to $0$ and the other, say $\gamma_2$, to $\frac{1}{z}.$
    \item[ii)] There are exactly two paths $\delta_1$ and $\delta_2$, of steepest ascent,  starting at the saddle point $t_0$ along which $Re \phi(t)$ increases,  both going to infinity. Together they split the plane into two (closed) simply-connected regions $D_0$ and $D_{1/z}$, with the property that if $t\in D_0$, there is a path of steepest descent from $t$ to $0$, and correspondingly for $D_{1/z}$.
     \end{enumerate}
     \end{lemma}
     \begin{proof} i) By the fact that $\phi''(t_0)=\frac{z^2(\al+1)^3}{\al}=|\phi''(t_0)|e^{i\beta}\neq 0,$ we know the local behavior of $\phi(t)$ at $t=t_0$(it will behave like $\beta t^2,\ \beta\in \mathbb C$ around the origin.) We see in particular that there are indeed locally two paths with $\text{Im} \phi(t)=\text{Im}  \phi(z_0)$ along which  $\text{Re} \phi(t)$ decreases. By the fact that $t_0$ is the only saddle point, these paths have to end up in either of the two points $0,1/z$, where $\text{Re} \phi(t)=-\infty$(the paths cannot vanish to infinity since we have that $\alpha=\eta+\zeta i, $ with $\eta>0$, so that $\lim_{t\to\infty}\text{Re}\phi(t)=\infty$.) It remains to check that they do not end up in the same point. They cannot intersect, again since there is only one saddle point.  By calculating the imaginary part $\text{Im}  \phi(t) $ around each branch point we find that there is a most one path {\it starting} at the branch point with constant $\text{Im}  \phi(t) $. Around $t=1/z$ the function will behave as $\log w+\delta$, for som $\delta\in \mathbb C$, and so there will at most be one ray corresponding to constant imaginary value. Around $t=0$, the function will behave as $t^\alpha$, with imaginary part $\zeta log\vert t\vert+\eta Arg t$, and this will be constant along a spiral approaching  $t=0$, (since $\zeta\neq 0$ by assumption) with a unique spiral corresponding to a fixed imaginary value. Hence the two paths have to end up at different points.

     ii) The second part follows similarily since the two paths cannot intersect themselves or each other, and since $\lim_{t\to\infty}\text{Re}\phi(t)=\infty$.
\end{proof}

 There is given $z\in \mathbb C$ two cases that may occur: either $1\in D_0 $ or $1\in D_{1/z}$. Define $E=\{ z: 1\in D_{1/z}\}$. In the case that $\alpha\in \mathbb R$, $z\in E$ is described by the condition $Re z\geq \al/(\al+1)$(see \cite{LDJG}), but we do not  have a good description of it for general $\alpha$. However we have an indirect description of it in terms of the descent picture of $w^\alpha(1-w)$. As above this function has a unique simple saddle point at $\alpha/(\al+1)$, and the paths of steepest ascent split the complex plane into two regions $\tilde D_0$ and $\tilde D_1$, containing the two branch points of the logarithm $0$ and $1$ respectively.
 \begin{lemma} $E=\tilde D_1$.
 \end{lemma}
\begin{proof} We use the notation of the preceding lemma. Now $z^\alpha g(t)=z^\alpha t^{\al }(1-tz)=w^\alpha(1-w)$, with $w=zt$(using a suitable determination of $\log$). There is then a path of steepest descent from $t=1$ to $t=0$, with respect to $g(t)$, if and only if there is a path of steepest descent from $w=z$ to $w=0$, with respect to $w^\alpha(1-w)$. In more detail, putting $\psi(w)=\log w^\alpha(1-w) $, we see that $z\psi'(zt)=\phi'(t)$, so that the tangent vectors of steepest descent paths are related to each other by multiplication by $\bar z$(compare the note just before the previous lemma). So if $\gamma(s),\ s\in [0,1]$ is a path from $z=\gamma(1)$ to $\gamma(0)=0$ such that the tangent $\gamma'(s)=\overline{\psi'(\gamma(s))}$, then $\delta(s)=z^{-1}\gamma(s)$  is a path from $1$ to $0$, such that $\delta'(s)=z^{-1}\overline{\psi'(z^{-1}\gamma(s))}=\vert z\vert ^{-2} \overline{\phi'(\delta(s))}$. Clearly $\delta$ will then be a parametrization of the path of steepest descent for $g(t)$. Hence  $1\in D_0\iff z\in \tilde D_0,$ which proves the lemma.
\end{proof}
We can see that $E^c$ contains the left half plane, using that $\eta>0$, where $\al=\eta+i\zeta$.\begin{lemma}
\label{lemma:principalbranch}
Suppose that $z=x+iy$ with $x\leq 0$. Then $z\in E^c$.
\end{lemma}
 \begin{proof}It suffices to show that the segment $\gamma(s)=sz, s\in [0,1]$ is an ascending curve relative the function $\psi(w)=w^\al(1-w)$.
 For $\gamma$ to be ascending from $0$, it is enough to show that $Re(\overline{\psi'(\gamma(s))\gamma'(s)})>0$, which after some simplifications is equivalent to $E(s):= \eta -2\eta xs+(-x+\eta (x^2+y^2))s^2+(x^2+y^2)s^3\geq 0$. This is a degree three curve in the plane and clearly $\lim_{s\to\infty}E(s)=\infty$, and it will have at most two stationary points, and between them an inflection point.
Now $E''(s)=0$ has the solution $s=\frac{x-\eta(x^2+y^2)}{3(x^2+y^2)}\leq 0$, by assumption. Since furthermore $E'(0)=-2\eta x\geq 0$ and $E(0)=n>0$, $E(s)$ will be strictly increasing with $s\geq 0$.

 \end{proof}

\subsection{Zero-free region}

 We sketch a proof that the complement to $E$ describes a zero-free region.

 \begin{proposition}
 \label{lemma:zero free} If $z\notin E$, then $z$ is not the limit of zeroes of $p_n(z)$.
 \end{proposition}
 \begin{proof}
(sketch) Assume that $z_n\to z$, is a sequence of zeros $p_n(z)=0$, where $z\notin E$. Hence $z_n\notin E$ for large enough $n$, and so there will be paths $\gamma_n$ of steepest ascent from $0$ to $1$, and this means, using standard Laplace-type techniques, that the integral
 $$
0=(\vert  \int_0^1 [t^{\al }(1-tz_n)]^ndt\vert )^{1/n},
 $$
 will have limit $\vert 1-z\vert $ (compare the proof of Lemma \ref{lemma:general}) which will be non-zero. This is a contradiction.
 \end{proof}

\subsection{The first integral} Now assume that $z\in E$. Then there is a path of steepest ascent from $1/z$ to $1$, and a path from $0$ to $1/z$, first as a path of steepest ascent to the saddle point, and then as a path of steepest descent to $1/z$.
We deform the integral accordingly to
\begin{equation}\label{deform}
\int_0^1 [g(t)]^ndt=\int_0^{\frac{1}{z}} [g(t)]^ndt+\int_{\frac{1}{z}}^1 [g(t)]^ndt=I_1+I_2
\end{equation}

We will now detail the path from $0$ to $1/z$ in the integral. Note that the spiral $\gamma_1$ around $0$ will intersect the line $L$ between $0$ and $1/z$ in an infinite set of points clustering towards $0$. Choose an arbitrary $\epsilon>0$. Start in $0$, walk along $L$  a distance at most $\epsilon $, until one of the intersection points $w_\epsilon$with $\gamma_1$ is reached. Call this part $L_\epsilon$, and continue along the part $\gamma_\epsilon$ of $\gamma_1$ that remains to $z_0$ and then to $1/z$ along $\gamma_2$. Starting with the principal branch of the logarithm that is defined $]0,1]$, at $z_0=\alpha/(\alpha+1)1/z$ and at $1/z$(by lemma \ref{lemma:principalbranch}), we may analytically continue this one along $\gamma_1$ to $w_\epsilon$; we interpret $t^\alpha:=e^{\alpha \log t}$ in this sense. Now if $\epsilon$ is small enough, $\vert1-zt\vert <2$, for $t\in L_\epsilon $, and so there exists $M>0$ such that $\vert t^{\al n}(1-zt)^{n}\vert <M t^\eta$ , for $t\in L_\epsilon $($M$ will depend on which branch of the logarithm we use). As a consequence, using that $\alpha=\eta+\zeta i$, with $\eta>0$,
the integral
\begin{equation}
\label{eq:Lepsilon}
\vert \int_0^{w_\epsilon}t^{\al n}(1-zt)^{n} dt\vert \leq \frac{M\epsilon^{\eta +1}}{\eta+1}.
\end{equation}

The integral between $w_\epsilon$ and $1/z$ is the one we will estimate with the saddle point method. Call it $I_1$. Note that the paths $\gamma_\epsilon$ and $\gamma_2$ are finite.

Now we rewrite the integral as

$$ \int_0^{\frac{1}{z}} [g(t)]^ndt=\int_0^{\frac{1}{z}}e^{n\phi(t)}dt,$$

where $\phi(z)=\log (t^\al(1-zt)).$ Recall that the saddle point was $z_0=\frac{\al}{(\al+1)z}$. As a consequence of the description of the paths of steepest descent, and standard techniques we may now deduce the value of the integral from Lemma \ref{prop:saddlepoint}.

\begin{lemma}
\label{lemma:saddlepoint} Let $z\neq 0$, and use the notation above. For any $\epsilon>0$,
\begin{equation}
\label{eq:firstintegral}
\vert \int_0^{\frac{1}{z}} [g(t)]^ndt\vert = g(z_0)^n\sqrt{\frac{2\pi}{n|\phi''(t_0)|}}+O(1/n)+\epsilon.
\end{equation}
\end{lemma}

\subsection{The Second integral}
Now we turn to the second integral
$$I_2(z)=\int_{\frac{1}{z}}^1[g(t)]^ndt,$$
 where $g(t)=t^\al(1-zt),$ and we assume that $z\in E$, so that we have a path of steepest ascent from $1/z$ to $1$. In particular, the path will never wind around $0$, so we may use the principal value determination of $\log$ for the definition of $t^\al$. We heuristically see that then the value $(1-z)^n$ at $t=1$ should dominate the integral asymptotically, and we will also provide details of a proof of this. We start with the following change of variables idea that we have taken from \cite{KDPD1}.

Letting  $$g(t)=t^\al(1-zt)=s(1-z), 0\leq s\leq 1,$$ we have that $s=0\Leftrightarrow t=\frac{1}{z}$ and $ s=1\Leftrightarrow t=1$
and that this definition implicitly defines  a path, say $\Delta,$ by  $t=t(s),$ between $\frac{1}{z}$ and $1$.
\begin{lemma}
\label{lemma:changepath}
$\int_{\frac{1}{z}}^1[g(t)]^ndt=\int_\Delta [g(t)]^ndt$
\end{lemma}
\begin{proof}
Since $$|g(t)|=O(|t|^\eta)~\text{and}~\eta>0,$$
the integral on a circle around $0$ goes to $0$ as the radius of the circle decreases. Hence we may move the path of integration freely in the complex plane.
\end{proof}
We have that
$$t^{\al-1}(\al-(\al+1)zt)dt=(1-z)ds.$$
Hence
$$I_2=(1-z)^n\int_0^1\frac{t(1-zt)s^{n-1}}{\al-(\al+1)zt}ds=(1-z)^nK(z)$$
\begin{equation}
\label{eq:second integral}
(I_2)^\frac{1}{n}=(1-z)(K(z))^\frac{1}{n}
\end{equation}
 
 The calculation of the limit of $I_2^\frac{1}{n}$ is completed by showing
 \begin{equation}\label{lim1}
 \lim_{n \to \infty}|K(z)|^\frac{1}{n}=1.
 \end{equation}
  Set
  $$ f(s)=\frac{t(1-zt)}{\al-(\al+1)zt}$$
so that
  $$k(z)=\int_0^1 f(s)s^{n-1}ds.$$
Then \eqref{lim1} is a consquence of the following general lemma.
   \begin{lemma}
   \label{lemma:general}
   Assume that $f(z)$ is a non zero function analytic in a domain $D,$ containing the unit interval $[0,1]$ on the real axis. Then
   $$\lim_{n\to \infty}\Bigg(\int_0^1f(s)s^{n-1}ds\Bigg)^{\frac{1}{n}}=1$$
   \end{lemma}
   \begin{proof}
   Let $I(n)=\int_0^1f(s)s^{n-1}ds.$ Since there is a positive number $M$, $0<M$ such that $|f(s)|\leq M$ if $s\in[0,1],$ we have that $$|I(n)|^{\frac{1}{n}}\leq M^{\frac{1}{n}},$$
and so
$$\limsup_{n\to \infty}|I(n)|^{\frac{1}{n}}\leq 1.$$
We may assume that $Re f(z)$ is not identically zero on $[0,1]$, since not both $Re f(z)$ and $Im f(z)$  can be identically zero on $[0,1]$ and multiplication by $i$ exchanges $Re f(z)$ and $Im f(z)$  and does not change $|I(n)|.$ We may also assume that $f(z)$ is not constant, since otherwise the result is immediate.
Then $\{z:Re f(z)=0\}$ has a finite intersection with$[0,1]$. Let $\delta$ be the largest element in this intersection which is strictly less than 1, so that $h(z):=Re f(z)$ is of constant sign in $[\delta,1)$.

We may assume that $h(s)> 0$ in $[\delta,1],$ if necessary, by multiplying $f(z)$ with $-1$.
\\ Now, $h(s)$ is bounded on $[0,1]$ so there is $M>0$ such that $-M\leq h(s)\leq M.$ Furthermore, there is to each $C:$
$$0<C<\max \{h(s), s\in[\delta,1]\},$$
an interval $[\delta_2,\delta_1]\subset [\delta,1)$ such that $\delta_2\leq s\leq \delta_1\Longrightarrow  h(s)\geq C.$ We may take $\delta_1$ maximal with this properties (if $h(1)\ne 0,$ then $\delta_1=1$ if $C $ is small enough.)
\\ Now estimate
$$Re(I(n))\geq 
\int_0^\delta h(s)s^{n-1}ds+\int_\delta^1h(s)s^{n-1}\geq \int_0^\delta-Ms^{n-1}ds+\int_{\delta_2}^{\delta_1}Cs^{n-1}ds$$
$$ =-M\frac{\delta^n}{n}+C\frac{\delta_1^n}{n}-C\frac{\delta_2^n}{n}
\\=\frac{\delta_1^n}{n}(-(\frac{\delta}{\delta_1})^nM-(\frac{\delta_2}{\delta_1})^nC+C).$$
Since $\frac{\delta}{\delta_1}<1$ and $\frac{\delta_2}{\delta_1}<1,$ there exists $n_0$ such that $n>n_0$
$$Re I(n)\geq \frac{\delta_1^n}{n}\frac{C}{2}\Longrightarrow \liminf|I(n)|^{\frac{1}{n}}\geq \delta_1$$
Now, $\delta_1$ depends on the choice of C and clearly $\lim_{c\to 0}\delta_1=1.$ Hence the lemma is proved.
  \end{proof}
  \subsection{Proof of the theorem}
  A zero $z_n$ of the polynomial $p_n(z)$ in (\ref{hyp}) satisfies, by  (\ref{deform}),   the equation $I_1=-I_2$.
  Supposing that we have a sequence of zeros $z_n\to z\in E$, we get that
  $$
  \lim_{n\to\infty}\vert I_1\vert^{\frac{1}{n}} = \lim_{n\to\infty}\vert I_2\vert^{\frac{1}{n}} =\vert 1-z\vert,
  $$
  by (\ref{lim1}) and (\ref{eq:second integral}).
 It follows from formula (\ref{eq2}), that  $z=0$ is not a zero of $p_n$.
  Hence Lemma \ref{lemma:steepest descent} and \ref{lemma:saddlepoint} apply, and we have that
 $$   \lim_{n\to\infty}\vert I_1\vert^{\frac{1}{n}} =\vert t_0^\alpha(1-zt_0)\vert.$$

 Hence $$
\vert(\frac{\al}{(\al+1)})^{\al}(\al+1)^{-1}\vert=\vert z^\al\vert \vert1-z\vert.
 $$
 (the form of the left hand side is chosen to avoid problems with changing the branch of the logarithm, note that $\frac{\al}{\al+1}$ is in the right half-plane, as well as $z$, by Lemma \ref{lemma:principalbranch}.)
This finishes the proof.

\bibliographystyle{plain}
\bibliography{xx}

  \end{document}